\documentclass[12pt,reqno]{amsart}
\usepackage{a4wide}
\usepackage{mathabx}
\usepackage{amsmath,amssymb}
\usepackage{xcolor}
\usepackage[colorlinks=true,linkcolor=blue,citecolor=blue]{hyperref}

\newtheorem{theorem}{Theorem}[section]
\newtheorem{proposition}{Proposition}[section]
\newtheorem{lemma}{Lemma}[section]
\newtheorem{remark}{Remark}[section]

\newtheorem*{conjecture*}{Conjecture} 

\newcommand{\N}{\mathbb{N}}
\newcommand{\R}{\mathbb{R}}

\newcommand{\wh}{\widehat}
\newcommand{\pa}{\partial}

\author{Klemens Fellner, Stefanie Sonner, Bao Quoc Tang, Do Duc Thuan}
\address{Klemens Fellner \hfill\break
Institute of Mathematics and Scientific Computing, University of Graz, Heinrichstrasse 36, 8010 Graz, Austria}
\email{klemens.fellner@uni-graz.at}
\address{Stefanie Sonner \hfill\break
Department of Mathematics, Radboud University Nijmegen, PO Box 9010, 6500 GL Nijmegen, The Netherlands}
\email{s.sonner@ru.nl}
\address{Do Duc Thuan \hfill\break
School of Applied Mathematics and Informatics, Hanoi University of Science and Technology, Dai Co Viet No. 1, Hanoi, Vietnam}
\email{thuan.doduc@hust.edu.vn}
\address{Bao Quoc Tang \hfill\break
Institute of Mathematics and Scientific Computing, University of Graz, Heinrichstrasse 36, 8010 Graz, Austria}
\email{quoc.tang@uni-graz.at} 

\title[Stabilisation by noise on the boundary]{Stabilisation by noise on the boundary for a Chafee-Infante equation with dynamical boundary conditions}

\begin{document}
\begin{abstract}
The stabilisation by noise on the boundary of the Chafee-Infante equation 
with dynamical boundary conditions subject to a multiplicative It\^o noise is studied. In particular, we show that there exists a finite range 
of noise intensities that imply the exponential stability of the trivial steady state. This differs from previous works on  
the stabilisation by noise of parabolic PDEs, where the noise acts inside the domain and
stabilisation typically occurs for an infinite range of noise intensities. 
To the best of our knowledge, this is the first result on the stabilisation of PDEs by  boundary noise.
\end{abstract}
\maketitle
\tableofcontents

\section{Introduction}\label{sec:Intro}
Let $D\subset\mathbb R^d$, $d \in\N$, be a bounded domain with smooth boundary $\partial D$
and denote by $Q_T = D\times (0,T)$, $S_T = \partial D\times (0,T)$ for any $T>0$. 
We consider the following stochastic Chafee-Infante 
equation with dynamical boundary conditions
\begin{equation}\label{eq}
	\begin{cases}
			d u + (- \Delta u + u^3 - \beta u)dt  = 0 &\text{in}\ Q_T,\\
			d u + (\partial_{\nu}u + \lambda u)dt = \alpha u\, dW_t &\text{on}\ S_T,\\
			u(x,0) = u_0(x), & x\in D,\\
			u(x,0) = \phi(x), & x\in \partial D,
	\end{cases}
\end{equation}
where $\beta>0$, $ \lambda>0$ and $\alpha \in \mathbb R$ are constants, and $\partial_\nu$ 
denotes the outward normal derivative on $\partial D$. Moreover, $W_t$ is a standard real-valued 
scalar Wiener process defined on the probability space $(\Omega, \mathcal F, P)$ with natural filtration 
$(\mathcal F_t)_{t\geq 0}$, $dW_t$ denotes the It\^ o differential and 
$(u_0,\phi)\in L^2(D)\times L^2(\partial D)$ are given initial data. 

The key feature of the model system~\eqref{eq} is the dynamical boundary condition in conjunction with the 
applied noise. There are many examples in the literature concerning partial differential equations with boundary noise
which are similar to \eqref{eq}. Let us mention an incomplete list of references concerning the well-posedness \cite{AB02,Bar11,Sow94}, stability \cite{AB02a}, or control \cite{DFT07,Mun17}.

More generally, parabolic equations with dynamical boundary conditions are used to model heat conduction 
in solids, for example when a solid is in contact with a well-stirred fluid at its surface, see e.g. \cite{CJ86}. For more results on derivation and analysis of dynamical boundary conditions, we refer the interested reader to e.g., \cite{Esc93,God06}.

\medskip

The main goal of this paper is to analyse whether a  \textit{multiplicative It\^o noise applied to} the boundary conditions will yield a stabilising effect for \eqref{eq}
compared to the deterministic problem $\alpha=0$.

\medskip
The stabilisation of PDEs by noise has been widely studied over the past decades, see e.g., \cite{CCLR07,CKS06,CLM01,Kwi99,Mao94}
and the references therein, as well as the extensive review \cite{CK09}.
Typically, the results strongly depend on the choice of interpretation, i.e. 
whether the SPDE is interpreted in the sense of It\^ o or Stratonovich. Moreover, also the type of noise is important, 
where mainly perturbations by additive and multiplicative noise are considered. 
Most works focus on the effect of a multiplicative It\^ o noise, less results have been obtained for Stratonovich 
noise, or SPDEs with additive noise.
Despite the extensive literature on the stabilisation of PDEs by noise, to the best of our knowledge, the stabilisation 
by a noise acting only on the boundary of a domain has not been addressed so far. 
This open problem and preliminary results were mentioned in  \cite{Car06, CK09}, but we are not aware of any published 
articles. 
This is the motivation for our work. As a first step in this direction we investigate the effect of a multiplicative 
It\^ o noise acting on the boundary of the domain on the Chafee-Infante 
equation with dynamical boundary conditions \eqref{eq}. More precisely, we derive sufficient conditions on the intensity of 
the noise $|\alpha|$, depending on $\beta,\lambda$ and the geometry of the domain that imply exponential 
stability of the trivial steady state solution $u\equiv 0.$ 

	
\medskip
Let us briefly describe the method and highlight the difficulties of the problem. 
With the use of classical Lebesgue- and Sobolev spaces $L^2(D)$,  $H^1(D)$ and  $H^{1/2}(\partial D)$ (see e.g. \cite{Adm}),
we introduce the Hilbert spaces
\begin{equation}\label{V0H}
		V_0 = H^1(D)\times H^{1/2}(\partial D), \qquad \quad H = L^2(D)\times L^2(\partial D)
\end{equation}
and 
\begin{equation}\label{V}
V=\{(u,T(u)): u\in H^ 1(D)\},
\end{equation}
where $T(u)=u|_{\partial D}$ denotes the trace operator $T\in\mathcal{L}(H^ 1(D);H^ {\frac{1}{2}}(\partial D)).$
Then, $V\subset V_0$ is a closed vector subspace, the embedding $V \hookrightarrow H$  is  dense and compact, and
$V \hookrightarrow H \hookrightarrow V^*$ is a Gelfand triple, 
where $V^*$ denotes the dual space of $V$.
The bilinear form 
$a: V\times V \to \mathbb R$ is defined by
	\begin{equation}\label{bilinear}
		a(U,\Phi) = \int_{D}\nabla u(x) \cdot \nabla \varphi(x) dx  + \lambda \int_{\partial D}u(x)|_{\partial D}\cdot \varphi|_{\partial D}(x) dS(x) \qquad \forall \, U,\Phi\in V,
	\end{equation}
where $U =( u, u|_{\partial D} )$ and $\Phi = 
( \varphi, \varphi|_{\partial D} )$. 
Since $a$ is symmetric, continuous, positive and coercive (due to the Poincar\'e-Trace inequality \eqref{crucial_ineq}, see Section \ref{sec:pre}), it defines a positive 
self-adjoint operator $A$ with compact resolvent  in $H$. 
Moreover,  introducing the operators $B: L^4(D)\times L^2(\pa D) \to  L^{\frac{4}{3}}(D)\times L^2(\pa D)$, 
$B(U) = (u^3 - \beta u, 0),$ 
and  $C:H\to H$, $C(U) = (0, \alpha u)$, 
we can rewrite problem \eqref{eq} in the abstract form
	\begin{align}\label{EqAb_0}
\begin{split}
dU + (AU + B(U))dt &= C(U)\, dW_t, \quad \text{ in } V^* + (L^{\frac 43}(D)\times L^2(\pa D))\\
U(0) &= (u_0, \phi)\in H.
\end{split}
	\end{align}
At this point, we remark that due to the ``degeneracy'' of the 
noise $C(U) = (0, \alpha u)$ which only acts on the second solution component, 
general results on the stabilisation by noise for abstract differential equations, e.g., see \cite{CLM01},
are not directly applicable to \eqref{EqAb_0}.
For the same reason, the stochastic system \eqref{eq} can also not be transformed into a random PDE, which would allow to 
apply deterministic methods, as done, e.g.  in 
\cite{CCLR07} for the stochastic Chafee-Infante equation with homogeneous Dirichlet boundary conditions.

Instead, to resolve these technical issues, we shall refine the method in \cite{CLM01}. 
The following Poincar\'e-Trace Inequality will be essential for our analysis: For any $\theta >0$, there exists an optimal constant $C_\theta^* >0$ such that 
\begin{equation}\label{crucial_ineq}
	\int_{D}|\nabla u(x)|^2dx + \theta\int_{\partial D}|u(x)|^2 dS(x) \geq C_{\theta}^*\int_{D}|u(x)|^2dx 
	\qquad \forall u\in H^1(D).
\end{equation}
The optimal constant $C_\theta^*$ is the first positive eigenvalue of the Laplacian $-\Delta$ in $D$ 
with Robin boundary conditions $\pa_\nu u + \theta u = 0$ on $\pa D$. To our knowledge, an explicit formula for $C_\theta^*$ is unknown. 
We remark that even in the limit $\theta \to \infty$, the constant $C_\theta^*$ remains bounded above by, for instance, the Poincar\'e constant for 
functions $u\in H_0^1(D)$, i.e. $C_\theta^* \leq \lambda_1$, where $\lambda_1$ denotes the first 
eigenvalue of the Laplacian $-\Delta$ in $D$ with homogeneous Dirichlet boundary conditions. In this work, 
in order to determine an explicit range of noise intensities for which stabilisation occurs, we derive an explicit expression for a sub-optimal
constant $C_\theta$ fulfilling \eqref{crucial_ineq}, which depends only on the dimension $d$ and the diameter of the domain $D$ (see Lemma \ref{TrIneq}). 

The first main result of this paper is Theorem \ref{pro}, which proves stabilisation by noise if the intensity $|\alpha|$ belongs to a specific and finite range. The proof uses the Poincar\'e-Trace inequality \eqref{crucial_ineq} to quantify the stabilising effect of the noisy dynamical boundary conditions to solutions of equation \eqref{eq}.
Our approach requires that the constant $\beta$ is strictly below a certain threshold. 
Though this restriction appears in our approach as a technical assumption, we conjecture due to the upper bound for the constant 
$C_\theta$ in \eqref{crucial_ineq}
that there indeed exists a critical value $\beta_{\mathrm{crit}}$ such that problem \eqref{eq} 
cannot be stabilised by noise on the boundary if $\beta > \beta_{\mathrm{crit}}$.

This is in contrast to the literature on the stabilisation by noise of parabolic SPDEs, where 
stabilisation typically occurs for an infinite range of noise intensities, i.e. whenever $|\alpha|$ is sufficiently large. It is an interesting open problem, whether the finite range of stabilising noise intensity is a characteristic of the stabilisation by noise on the boundary, 
or it is due to the imposed dynamical boundary conditions, or a technical
limitation of our method.

\medskip
To further investigate this question, we compare the boundary noise problem \eqref{eq}
to the following Chafee-Infante equation with a multiplicative It\^ o noise acting \textit{inside the domain}
and subject to noise free dynamical boundary conditions
	\begin{equation}\label{eq_0}
		\begin{cases}
			d u + (- \Delta u + u^3 - \beta u)dt  = \alpha u\, dW_t &\text{in}\ Q_T,\\
			d u + (\partial_{\nu}u + \lambda u)dt = 0 &\text{in}\ S_T,\\
			u(x,0) = u_0(x), & x\in D,\\
			u(x,0) = \phi(x), & x\in \partial D.
		\end{cases}
	\end{equation}
By applying a similar approach as for problem \eqref{eq}, we obtain again 
a finite range of noise intensities $|\alpha|$ that yields stabilisation of the equation. This might suggest that the dynamical 
boundary conditions \textit{per se} prevent the stabilisation of the Chafee-Infante equation by noise with too large intensities 
(if the noise acts either on the boundary or inside the domain, but not simultaneously on both parts).  
Interestingly, for certain parameter regimes, our method of proof leads to a significant difference between the problems
\eqref{eq_0} and \eqref{eq}. 
Namely, in the case that $\lambda > \beta$ one can 
always stabilise equation \eqref{eq_0} by noise with suitable intensity, no matter how large $\beta$ is, whereas 
for problem \eqref{eq} we can only show stabilisation for $\beta$ below a critical value $\beta_{\mathrm{crit}}$.

\medskip
We remark that the parameter conditions of our results imply the \textit{preservation of stability}, i.e. 
if the trivial steady state of the deterministic problem is stable, then, its stability is preserved under  
stochastic perturbations by noise with sufficiently small intensity. On the other hand, we also show 
\textit{stabilisation by noise}, i.e. for  parameter ranges for which the zero solution of the deterministic 
equation is unstable, 
adding noise on the boundary with an appropriate intensity leads to stabilisation.

To highlight this latter case, we analyse problem \eqref{eq} in one dimension in more details in the last section of the paper. Firstly, we study the linearised equation around the zero 
steady state of the deterministic problem, 
where we can derive an explicit representation for the solution by separation of variables. Then, by choosing appropriate 
values for the parameters $\beta >0$ and $\lambda>0$ we 
show that the zero solution of the linearised equation is unstable.  
Since all the eigenvalues of the stationary problem have non-zero real part, an infinite dimensional version of the Hartman-Grobman theorem implies the instability of the zero solution of the nonlinear deterministic Chafee-Infante model. 
Finally, applying our main results on stabilisation by noise allows to determine an explicit range of noise intensities that stabilise
 equation \eqref{eq}.

\medskip 
{\bf Organisation of the paper:} In Section \ref{sec:pre}, we prove the crucial functional inequality \eqref{crucial_ineq}, 
derive an explicit representation for the constant $C_\theta$
and analyse the stability of the zero solution for the unperturbed deterministic problem. 
The main result on the stabilisation by noise on the boundary is established in Section \ref{sec:Stochastic}. 
In particular, we determine an explicit range for the noise intensities for \eqref{eq} that lead to
stabilisation. 
Section \ref{sec:Stochastic2} is devoted to problem \eqref{eq_0}, where the noise acts inside the domain, 
and the results on stabilisation are compared to the setting with boundary noise \eqref{eq}.  
In section \eqref{sec:1D}, a one-dimensional example the 
instability of the deterministic problem and the stabilisation by noise on the boundary is detailed. And finally Section \ref{sec:concl} is a conclusion.

\section{Preliminaries}\label{sec:pre}

{\bf Notations:} Here and in the sequel, we denote by $\|\cdot\|_{D}$ and $\|\cdot\|_{\partial D}$ 
the norms in $L^2(D)$ and $L^2(\partial D)$, respectively. 
The inner products in $L^ 2(D)$ and $L^2(\partial D)$ are denoted by $\langle\cdot,\cdot\rangle_D$ and $\langle\cdot,\cdot\rangle_{\partial D}$, and 
the norm in $L^p(D)$ for $p\ne 2$ by $\|\cdot\|_{p,D}$.

\subsection{A Poincar\'e-Trace inequality}

The following inequality and the explicit expression for the corresponding constant play an important role in our analysis.
	\begin{lemma}[A Poincar\'e-Trace Inequality]\label{TrIneq}
	Let $D\subset\mathbb R^d$, $d \in\N$, be a bounded domain with smooth boundary $\partial D$. 
		For every $\theta > 0$, there exists an optimal constant $C_\theta^*>0$ such that 
		\begin{equation}\label{func_ineq}
		 C_{\theta}^*\|u\|_{D}^2 \leq \|\nabla u\|_D^2 + \theta \|u\|_{\partial D}^2 \qquad \forall u\in H^1(D),
		\end{equation}
		where $C_{\theta}^*$ is continuous and non-increasing with respect to $\theta$, and $C_0^*=0$. 
		
		\medskip
		Let $R = \frac 12 \mathrm{diam}(D)$. Then the constant $C_\theta$ defined as
		\begin{equation}\label{C_theta}
			C_\theta = \begin{cases}
				\theta\left(d/R - \theta\right), &\text{ if } \theta \in \left[0,\frac{d}{2R}\right),\\
				\frac{d^2}{4R^2}, &\text{ if } \theta \in \left[\frac{d}{2R},\infty\right),
				\end{cases}
		\end{equation}
		also fulfils \eqref{func_ineq}.
	\end{lemma}
	
\begin{remark}\label{remark1}\hfill
	\begin{itemize}
		\item[(i)] For the rest of this paper, we denote by $C_\theta^*$ the optimal constant in \eqref{func_ineq} while $C_\theta$ is the explicit constant defined in \eqref{C_theta}.
		\item[(ii)] Inequality \eqref{func_ineq} is well-known in functional analysis concerning equivalent 
		norms in $H^1(D)$, see e.g., \cite{SM09}. However, its proof is usually based on a contradiction argument and 
		thus, does not yield an explicit expression nor a quantitative estimate for $C_\theta^*$. $C_\theta^*$ is the first eigenvalue of the Laplace operator with Robin boundary conditions $\pa_\nu u + \theta u = 0$, see e.g. \cite{Fil15,Kov14}. 
		In one dimension, one can solve  the eigenvalue problem explicitly. Obtaining the optimal constant $C_\theta^*$ in higher dimensions goes beyond the scope of this paper.
		\item[(iii)] As $\theta$ varies in $(0,\infty)$, $C_\theta^*$ always has an upper bound $C_{\theta,\max}^*$. Indeed, choosing $u\in H_0^1(D)$ it follows from Poincar\'e's inequality that $C_\theta^* \leq \lambda_1$ for any $\theta>0$, where $\lambda_1>0$ is the first  eigenvalue of the Laplacian $-\Delta$ with homogeneous Dirichlet boundary conditions.
		\item[(iv)] The explicit expression for $C_\theta$ in \eqref{C_theta} allows us to specify ranges of noise intensities, depending on $\alpha, \beta,d$ and $R$, that stabilise the equation.
	\end{itemize}
\end{remark}

\begin{proof}[Proof of Lemma \ref{TrIneq}]
Consider the eigenvalue problem for the Laplacian with Robin boundary conditions
\begin{equation*}
\begin{cases}
-\Delta u = \kappa u, &\text{ in }  D,\\
\pa_\nu u + \theta u = 0, &\text{ on } \pa D.
\end{cases}
\end{equation*}
By classical results from the calculus of variations, the first eigenvalue is given by
\begin{equation*}
C_\theta^*:= \inf_{u\in H^1(D), u\ne 0}\frac{\|\nabla u\|_{D}^2 + \theta\|u\|_{\pa D}^2}{\|u\|_D^2},
\end{equation*}
it is positive and the corresponding first eigenfunction $\psi_1\in H^1(\Omega)$ is positive. Then, certainly, 
\eqref{func_ineq} is satisfied. The continuity of $C_\theta^*$ with respect to $\theta$ follows from \cite[Theorem 1]{Fil14} ($C_\theta^*$ is even differentiable w.r.t $\theta$).

\medskip
We now show that the constant $C_\theta$ defined in \eqref{C_theta} also fulfills inequality 
\eqref{func_ineq}. Since $2R = \mathrm{diam}(D)$ we can choose $x_0\in\mathbb{R}^d$ 
such that $|x-x_0|\leq R$ for all $x\in \partial{D},$
where $|\cdot|$ denotes the norm in $\R^d.$ 
We consider the function $\Phi(x)=\frac{|x-x_0|^2}{2d}$. Then, 
$\nabla \Phi(x)=\frac{x-x_0}{d}$ and $\Delta\Phi(x)=1$. Let $\theta\in(0,\frac{d}{R})$. 
Then, using integration by parts and Young's inequality it follows that 
\begin{align*}
\int_Du^2(x)dx&=\int_Du^2(x)\Delta \Phi(x) dx\\
&=\int_{\partial D}u^2(x)\partial_\nu\Phi(x) dS(x)-2\int_D u(x)\nabla u(x)\cdot \nabla\Phi(x) dx\\
&\leq \|\partial_\nu\Phi\|_{\infty,\partial D}\|u\|^ 2_{\partial D}+ \|\nabla\Phi\|_{\infty,D}
\left(\theta \|u\|^ 2_D+\frac{1}{\theta}\|\nabla u\|^2_{D}\right).
\end{align*}
Hence, we obtain 
\begin{align*}
\left(1-\theta \|\nabla\Phi\|_{\infty,D}\right)\|u\|^ 2_D&
\leq \|\partial_\nu\Phi\|_{\infty,\partial D} \|u\|^ 2_{\partial D}+ \frac{\|\nabla\Phi\|_{\infty,D}}{\theta}\|\nabla u\|^2_{D},
\end{align*}
which implies that 
$$
\theta\left( \frac{1}{\|\nabla\Phi\|_{\infty,D} }-\theta\right)\|u\|^2_D\leq 
\theta\frac{\|\partial_\nu\Phi\|_{\infty,\partial D} }{\|\nabla\Phi\|_{\infty,D}} \|u\|^2_{\partial D}+ \|\nabla u\|_D^2.
$$
Now, using that
$$
\|\nabla\Phi\|_{\infty,D}=\frac{R}{d} \qquad \text{ and } \qquad \|\partial_\nu\Phi\|_{\infty,\partial D} = \|\nabla \Phi \cdot \nu \|_{\infty, \partial D}\leq \frac{R}{d}
$$
we obtain 
$$
\theta\left(\frac{d}{R}-\theta\right)\|u\|_{D}^2\leq \|\nabla u\|_D^2 + \theta \|u\|_{\partial D}^2.
$$
As required, $C_0=0$ and $C_\theta:= \theta(\frac dR - \theta)$ is increasing for $\theta\in\left(0,\frac{d}{2R}\right]$.
However, since $\theta(\frac{d}{R} - \theta)$ is decreasing within the interval $\left(\frac{d}{2R}, \frac{d}{R}\right)$, 
we observe that for $\theta\geq \frac{d}{2R}$
\begin{equation*}
	\|\nabla u\|_D^2 + \theta \|u\|_{\partial D}^2 \geq \|\nabla u\|_D^2 + \frac{d}{2R}\|u\|_{\partial D}^2 \geq \frac{d^2}{4R^2}\|u\|_{D}^2
	\qquad \forall u\in H^ 1(D),
\end{equation*}
and hence, we set $C_\theta=\frac{d^2}{4R^2}$ for all $\theta \in \left[\frac{d}{2R},\infty\right)$.
Then, $C_\theta$ is non-decreasing and depends continuously on $\theta$, which completes the proof.
\end{proof}

\subsection{The deterministic equation}\label{sec:Deterministic}

In this section, we consider the unperturbed deterministic Chafee-Infante  equation with dynamical boundary conditions, 
	\begin{equation}\label{DeEq}
		\begin{cases}
			u_t - \Delta u + u^3 - \beta u  = 0 &\text{in}\ Q_T,\\
			u_t + \partial_{\nu}u + \lambda u = 0 &\text{on}\ S_T,\\
			u(x,0) = u_0(x), & x\in D,\\
			u(x,0) = \phi(x), & x\in \partial D,
		\end{cases}
	\end{equation}	
	and formulate sufficient conditions for the exponential stability of the trivial steady state.

As in the introduction 
we first rewrite problem \eqref{DeEq} in an abstract form. 
To this end let  
	$$
		H = L^2(D)\times L^2(\partial D),\qquad V_0=H^1(D)\times H^{1/2}(\partial D)
	$$
be as defined in the Introduction with the norms $\|\cdot\|^2_H=\|\cdot\|^2_D+\|\cdot\|^2_{\partial D}$ and 
	$\|\cdot\|^2_{V_0}=\|\cdot\|^2_{H^ 1(D)}+\|\cdot\|^2_{H^{\frac{1}{2}} (\partial D)}$ and 
	the inner product in $H$ be denoted by
	$\langle\cdot,\cdot\rangle=\langle\cdot,\cdot\rangle_{H}=\langle\cdot,\cdot\rangle_{ D}+\langle\cdot,\cdot\rangle_{\partial D}$.
	Moreover, let
	$$	
		 V = \{(u,T(u)): u\in H^ 1(D)\}\subset V_0, 
	$$
with the norm induced by $V_0$, where $T\in\mathcal{L}(H^ 1(D);H^ {\frac{1}{2}}(\partial D))$ denotes the trace operator $T(u)=u|_{\partial D}.$ 
Then, $V$ is a closed vector subspace of $V_0$, and densely and compactly embedded into $H$. 
Identifying $H$ with its dual we have the Gelfand triple $V\hookrightarrow H\hookrightarrow V^ *$, 
where $V^*$ denotes the dual of $V$. 
Let $A:V\to V^*$ be the continuous linear operator 
defined by  the symmetric, continuous bilinear form 
\begin{equation*}
		a(U, \Phi) = \int_{D}\nabla u(x) \cdot \nabla \varphi(x) dx + \lambda \int_{\partial D}u(x)|_{\partial D} \varphi(x)|_{\partial D} dS(x) 
		\qquad U,\Phi\in V,
	\end{equation*}
where $U = ( u,u|_{\partial D} )$, $\Phi = (\varphi, \varphi|_{\partial D})$. 
The Poincar\'e-Trace inequality \eqref{func_ineq} implies that $a$ is coercive and hence, by the Lax-Milgram Theorem $A$ has a 
bounded inverse $A^ {-1}:V^ *\to V$. Its restriction to $H$ is a compact bounded operator 
and its inverse is the operator $A:D(A)\to H$ with domain $D(A)=\{U\in V: AU\in H\}$.
Moreover, Lemma \ref{TrIneq} implies that 
$A$  is positive,
\begin{align*}
\langle AU,U\rangle &=a(U, U)= \int_{D}|\nabla u(x)|^ 2 dx + \lambda \int_{\partial D}|u(x)|_{\partial D}|^ 2 dS(x) \\
&\geq C_{\frac{\lambda}{2}}^*\|u\|_D^ 2+\frac{\lambda}{2}\|u\|^ 2_{\partial D}\geq 
\min\left\{\frac{\lambda}{2},C_{\frac{\lambda}{2}}^*\right\}\|U\|_H^2&&\forall U\in D(A), 
\end{align*}
where $C_{\frac{\lambda}{2}}^*$ is the optimal constant in \eqref{func_ineq}.
Hence, since $A$ is a positive, self-adjoint operator with compact resolvent,
there exists an orthonormal basis $\{V_j\}\subset D(A)$ in $H$ consisting of 
eigenfunctions of $A$ with corresponding eigenvalues $\lambda_j $ such that 
$$
0<\lambda_j\leq \lambda_{j+1},\ j\in\N, \qquad \lambda_j\to\infty.
$$ 
Defining $B: L^4(D)\times L^2(\partial D) \to L^{\frac{4}{3}}(D)\times L^2(\partial D)$, 
$B(U) = (u^3 - \beta u, 0),$ 
we rewrite \eqref{DeEq} as
	\begin{align}\label{DeEq1}
		\begin{split}
			\frac{d}{dt}U + AU + B(U) &=0, \quad \text{ in } V^*+(L^{\frac 43}(D)\times L^2(\pa D)), \\
			U(0) = U_0 &= (u_0, \phi)\in H.
		\end{split}
	\end{align}

With an abuse of notation, in the sequel we will drop the amendment $|_{\partial D}$ in the second component of $U$, 
where the value of the function on the boundary is taken. 
We observe that
	\begin{align}\label{AB}
		\langle A U, U \rangle  &= \|\nabla u\|_D^2 + \lambda\|u\|_{\partial D}^2,&& U\in D(A),\nonumber\\
		\langle B(U), U \rangle &= \|u\|_{4,D}^4 - \beta \|u\|_D^2,&& U\in V.
	\end{align}	
	\begin{theorem}\label{determin}
		For any initial data $(u_0, \phi)\in L^2(D)\times L^2(\partial D)$ and $T>0$, there exists a unique weak 
		solution $u$ of \eqref{DeEq}, and 
		\begin{equation*}
			u\in C([0,T]; L^2(D)) \cap L^2([0,T];H^1(D))\cap L^4([0,T];L^4(D)),
		\end{equation*}
		\begin{equation*}
			u|_{\partial D} \in C([0,T]; L^2(\partial D))\cap L^2([0,T];H^{1/2}(\partial D)).
		\end{equation*}
		
		Moreover, if $\lambda$ and $\beta$ are such that 
		\begin{equation*}
			C_{\lambda}^* > \beta,
		\end{equation*}
		then the zero steady state  is exponentially stable. On the other hand, if 
		\[
			C_\lambda^* < \beta
		\]
		the zero steady state is unstable.
	\end{theorem}
	\begin{proof}
		The existence and uniqueness of solutions follows as, e.g. in \cite{CMR17}, pp. 824--825, see also \cite[Theorem 1.4]{Li69}
		and the subsequent remark.

		\medskip		
		To show the exponential stability of the zero steady state, we compute
		\begin{equation*}
			\begin{aligned}
			\frac 12 \frac{d}{dt}\|u\|_D^2 &= \int_{D}u(x) u_t(x)dx = \int_{D}u(x)(\Delta u(x) - u^3(x) + \beta u(x))dx\\
			&= -\|\nabla u\|_D^2 + \int_{\partial D}u(x)\partial_{\nu}u(x) dS(x) - \|u\|_{4,D}^4 + \beta \|u\|_{D}^2\\
			&= -\|\nabla u\|_D^2 - \frac 12 \frac{d}{dt}\|u\|_{\partial D}^2 - \lambda \|u\|_{\partial D}^2 - \|u\|_{4,D}^4 + \beta \|u\|_{D}^2,
			\end{aligned}
		\end{equation*}
		and therefore,
		\begin{equation*}			
			\frac 12 \frac{d}{dt}(\|u\|_D^2 + \|u\|_{\partial D}^2) \leq -(\|\nabla u\|_D^2 + \lambda \|u\|_{\partial D}^2 - \beta\|u\|_D^2).
		\end{equation*}
		Since $C_{\lambda}^* > \beta$ and $C_{\lambda}^*$ depends continuously on $\lambda$ (see Lemma \ref{TrIneq}), 
		there exists $\varepsilon>0$ such that $C_{\lambda - \varepsilon}^* > \beta$. Thus, we obtain the estimate
		\begin{equation*}
			\|\nabla u\|_{D}^2 + \lambda \|u\|_{\partial D}^2 = \varepsilon \|u\|_{\partial D}^2 + (\|\nabla u\|_D^2 + (\lambda - \varepsilon)\|u\|_{\partial D}^2) \geq \varepsilon\|u\|_{\partial D}^2 + C_{\lambda-\varepsilon}^*\|u\|_{D}^2,
		\end{equation*}
		and it follows that
		\begin{equation*}
			\frac 12 \frac{d}{dt}(\|u\|_D^2 + \|u\|_{\partial D}^2) \leq - \varepsilon\|u\|_{\partial D}^2 - (C_{\lambda - \varepsilon}^* - \beta)\|u\|_{D}^2 = -\delta(\|u\|_D^2 + \|u\|_{\partial D}^2),
		\end{equation*}
		where $\delta = \min\{\varepsilon, C_{\lambda - \varepsilon}^* - \beta\} > 0$. Finally, we get
		\begin{equation*}
			\|u(t)\|_D^2 + \|u(t)\|_{\partial D}^2 \leq e^{-2\delta t}(\|u_0\|_{D}^2 + \|\phi\|_{\partial D}^2) \qquad \forall t\geq 0,
		\end{equation*}
		which proves the exponential stability of the zero steady state.
		
		\medskip
		To prove instability, we use Kaplan's method. It suffices to show the instability for the linearised problem. By an infinite dimensional Hartman-Grobman theorem, see e.g. \cite{Lu91} or \cite[Corollary 5.1.6]{Hen81}, it then follows the instability for the nonlinear problem. Since $\beta > C_\lambda^*$ and $C_\lambda^*$ depends continuously on $\lambda$, we can choose some $\theta > \lambda$ such that $\beta > C_\theta^*$. Denote by $\psi_1$ the positive eigenfunction corresponding to the first eigenvalue $C_\theta^*$, i.e. $-\Delta \psi_1 = C_\theta^*\psi_1$ in $D$ and $\pa_\nu \psi_1 + \theta \psi_1 = 0$ on $\pa D$. 
		Testing the linear equation 
		\begin{equation*}
			\begin{cases}
				u_t - \Delta u - \beta u = 0, &\text{ in } D,\\
				u_t + \pa_\nu u + \lambda u = 0, &\text{ on } \pa D
			\end{cases}
		\end{equation*}
		with $\psi_1$ and integrating by parts, we obtain
		\begin{equation*}
			\begin{aligned}
			\frac{d}{dt}\int_D u\psi_1 dx &= \int_D \psi_1 (\Delta u + \beta u)dx\\
			&= -\frac{d}{dt}\int_{\pa D}u\psi_1 dS + \beta\int_{D}u\psi_1 dx - \lambda\int_{\pa D}u\psi_1dS + \int_D u\Delta \psi_1 dx + \theta \int_{\pa D}u\psi_1dS.
			\end{aligned}
		\end{equation*}
		By using $\Delta \psi_1 = -C_\theta^*\psi_1$, we have
		\begin{equation*}
			\begin{aligned}
			\frac{d}{dt}\left(\int_D u\psi_1 dx + \int_{\pa D}u \psi_1 dS\right) &= (\beta - C_\theta^*)\int_D u\psi_1 dx + (\theta - \lambda)\int_{\pa D}u\psi_1 dS\\
			&\geq \min\{\beta - C_\theta^*; \theta - \lambda \}\left(\int_D u\psi_1 dx + \int_{\pa D}u \psi_1 dS\right),
			\end{aligned}
		\end{equation*}
		which implies the exponential growth of $\int_D u\psi_1dx + \int_{\pa D}u\psi_1 dS$, and consequently the instability of the zero steady state.
		
	\end{proof}
\section{Stabilisation by noise on the boundary}\label{sec:Stochastic}
%

In this section, we formulate an existence result for strong solutions of the stochastic Chafee-Infante equation \eqref{eq}, 
and investigate whether the equation can be  stabilised by noise on the boundary.

\medskip

As in the introduction 
we first rewrite problem \eqref{eq} in the abstract form 
	\begin{align}\label{EqAb}
		\begin{split}
			dU + (AU + B(U))dt &= C(U)\, dW_t,\\
			U(0) = U_0 &= (u_0, \phi)\in H,
		\end{split}
	\end{align}
where the operators $A$ and $B$ were defined in the previous section, and $C:H\to H$
is given by $C(U) = (0, \alpha u)$.
We will frequently use the identities \eqref{AB} and 
	\begin{equation*}\label{oper-estimate}
		\langle C(U), U \rangle  = \alpha\|u\|_{\partial D}^2, \qquad  U\in H.\nonumber
	\end{equation*}

The  existence and uniqueness of strong solutions of \eqref{EqAb} follows, e.g., from 
a slight modification of \cite[Theorem 4.1]{Par79}.
\begin{theorem}
		Let  $T>0$ and assume that the initial data $U_0 = (u_0, \phi)$ is an $H$-valued 
		$\mathcal F_0$-measurable random variable satisfying $\mathbb{E}\|U_0\|_{H}^2<\infty$. 
		Then, there exists a unique strong solution of \eqref{EqAb} in $C([0,T]; H)$ such that
		\begin{equation*}
			\mathbb{E}\Big(\sup_{t\in [0,T]}\|U(t)\|_{H}^2\Big) < \infty \quad \text{ and } \quad \
			\mathbb{E}\Big(\int_0^T\|U(t)\|_{V}^2dt \Big)< \infty.
		\end{equation*}
\end{theorem}

As shown in Section \ref{sec:Deterministic}, in the deterministic case the zero steady state is exponentially stable
provided that $\beta< C_{\lambda}^*$, while the stability is lost for $\beta> C_{\lambda}^*$. 
We now investigate the stability of the zero solution when the system is perturbed by a multiplicative It\^o noise $\alpha u \, dW_t$ 
on the boundary.  In particular, we show that if $\beta<C_\lambda^*$ the exponential stability of the zero steady state  
is preserved for small noise intensities $|\alpha|$ and,  if $\beta\geq C_{\lambda}^*$, 
the zero steady state can be stabilised by noise on the boundary for certain parameter ranges of $\lambda$ 
and $\beta$.

\medskip
Our main result is the following, it yields sufficient conditions for the exponential 
stability of the trivial steady state of the stochastic problem \eqref{EqAb}.

\begin{theorem}\label{pro}
Let $C_\theta^*$ denote the optimal constant in Lemma \ref{TrIneq} and the initial data $U_0\in L^ 2(\Omega,\mathcal{F}_0,P;H)$ 
be such that $\|U_0\|_H\neq 0$ $P$-a.s..

	If there exists a constant $\theta >0$ such that either
	\begin{equation*}
		 C_\theta^* - \beta > \theta - \lambda > 0,
	\end{equation*}
	or
	$$
	C_\theta^* > \beta\quad \text{ and } \quad\lambda\geq \theta,
	$$ 
	then, the solution of \eqref{EqAb} satisfies
	\begin{equation*}
		\limsup_{t\to \infty}\frac{1}{t}\log\|U(t)\|_H^2 < 0\qquad P\text{-a.s.}
	\end{equation*}
	for all noise intensities $|\alpha|$ such that 
	\begin{equation*}
		\frac{\alpha^2}{2} \in
		\begin{cases}
			[\max\{0,\theta - \lambda\}, Z_1) &\text{ if }\ C_\theta^* - \beta > 2(\theta - \lambda),\\
			(Z_1, Z_2) &\text{ if }\ 2(\theta - \lambda) \geq C_\theta^* - \beta > \theta - \lambda,
		\end{cases}
	\end{equation*}
	where
	\begin{equation*}
		Z_{1,2} :=  3(C_\theta^* - \beta) + \lambda - \theta \mp 2\sqrt{2(C_\theta^* - \beta)(C_\theta^* - \beta + \lambda - \theta)}
	\end{equation*}
	with $Z_1$ and $Z_2$ are corresponding to the sign $-$ and $+$ respectively.
\end{theorem}
\begin{proof}
We will apply It\^o's formula (e.g., see \cite{Par79}) to $\psi(U) = \log\|U\|_H^2$, 
where we assume (w.l.o.g.) that $\|U\|_H^2\neq 0$ since otherwise $\|U\|_{H} = 0$ which implies $u=0$ a.e. in $D$ and thus no further stabilisation is needed.
In applying It\^o's formula, we observe that 
	the Fr\'echet derivatives  $\psi'$ and $\psi''$ can be expressed as 
	\begin{equation*}
		\psi'(U) = \frac{2\langle U, \cdot \rangle}{\|U\|_H^2} \quad \text{ and } \quad 
		\psi''(U)h = \frac{2\langle h, \cdot \rangle}{\|U\|_H^2} - \frac{4\langle U, h\rangle \langle U, \cdot \rangle}{\|U\|_H^4}, \qquad U,h\in H.
	\end{equation*}
	Hence, we obtain 
	\begin{equation}\label{e1}
	\begin{aligned}
	\log \|U(t)\|_H^2 &= \log \|U(0)\|_H^2 - 2\int_0^t\frac{1}{\|U\|_H^2}\left(\langle AU + B(U), U\rangle - \frac 12\|C(U)\|_H^2\right)ds\\
	&\quad - 2\int_0^t\frac{\langle C(U), U\rangle^2}{\|U\|_H^4}ds + 2\int_0^t\frac{\langle C(U), U\rangle}{\|U\|_H^2}dW_s\\
	&= \log \|U(0)\|_H^2 - 2\int_0^t\frac{\|\nabla u\|_D^2 + \lambda \|u\|_{\partial D}^2 + \|u\|_{4,D}^4 - \beta \|u\|_D^2 - \frac 12\alpha^2\|u\|_{\partial D}^2}{\|U\|_H^2}ds\\
	&\quad - 2\int_0^t\frac{\langle C(U), U\rangle^2}{\|U\|_H^4}ds + 2\int_0^t\frac{\langle C(U), U\rangle}{\|U\|_H^2}dW_s.
	\end{aligned}
	\end{equation}
	In order to estimate the stochastic integral we apply the exponential martingale inequality (see e.g., \cite[Lemma 1.1]{Li97}) 
	and obtain
	\begin{equation}\label{martingale}
	P\left\{\omega: \sup_{0\leq t\leq T}\left[\int_0^t\frac{\langle C(U), U\rangle}{\|U\|_H^2}dW_s - \frac{\kappa}{2}\int_0^t\frac{\langle C(U), U\rangle^2}{\|U\|_H^4}ds\right] > \frac{2\log k}{\kappa}\right\} \leq \frac{1}{k^2},
	\end{equation}
	where $\kappa \in (0,1)$ and $k\in\N.$ Now, the Borel-Cantelli lemma implies that there exists 
	$k_0( \omega) >0$ for almost all $\omega\in\Omega$ such that
	\begin{equation*}
	\begin{aligned}
	\int_0^t\frac{\langle C(U), U\rangle}{\|U\|_H^2}dW_s 
	&\leq \frac{\kappa}{2}\int_0^t\frac{\langle C(U), U\rangle^2}{\|U\|_H^4}ds +  \frac{2\log k}{\kappa} 
	=\frac{\kappa}{2}\int_0^t\frac{\alpha^2\|u\|_{\partial D}^4}{\|U\|_H^4}ds +  \frac{2\log k}{\kappa}
	\end{aligned}
	\end{equation*}
	for all $k\geq k_0(\omega)$.
	Inserting this estimate into \eqref{e1} yields
	\begin{equation}\label{e1_0}
	\begin{aligned}
	&\quad \log \|U(t)\|_H^2\\
	&\leq \log \|U(0)\|_H^2 + \frac{4\log k}{\kappa}- (2 - \kappa)\int_0^t\frac{\alpha^2\|u\|_{\partial D}^4}{\|U\|_H^4}ds\\
	&\quad - 2\int_0^t\frac{\|\nabla u\|_D^2 + \lambda \|u\|_{\partial D}^2 + \|u\|_{4,D}^4 - \beta \|u\|_D^2 - \frac 12\alpha^2\|u\|_{\partial D}^2}{\|U\|_H^2}ds \\
	&\leq \log \|U(0)\|_H^2 + \frac{4\log k}{\kappa} + 2\beta t\\
	&\quad - \int_0^t\frac{(2\|\nabla u\|_D^2 + 2\lambda \|u\|_{\partial D}^2 + 2\beta \|u\|_{\partial D}^2 - \alpha^2\|u\|_{\partial D}^2)\|U\|_H^2 + (2-\kappa)\alpha^2\|u\|_{\partial D}^4}{\|U\|_H^4}ds.
	\end{aligned}
	\end{equation}
	It is now sufficient to show that there exists a range of $\alpha\in\R$ such that 
	\begin{equation}\label{aim}
	(2\|\nabla u\|_D^2 + 2\lambda \|u\|_{\partial D}^2 + 2\beta \|u\|_{\partial D}^2 - \alpha^2\|u\|_{\partial D}^2)\|U\|_H^2 + (2-\kappa)\alpha^2\|u\|_{\partial D}^4 > 2\beta \|U\|_H^4.
	\end{equation}
	Indeed, if \eqref{aim} holds, then \eqref{e1_0} implies that 
	\begin{equation*}
	\limsup_{t\to \infty}\frac{1}{t}\log\|U(t)\|_H^2 < \limsup_{t\to\infty}\frac{1}{t}\left(\log \|U(0)\|_H^2 + \frac{4\log k}{\kappa}\right) = 0,
	\end{equation*}
	which is the statement of the theorem. 
	
	We now prove \eqref{aim}. 
	By Lemma \ref{TrIneq} for every $\theta >0$ there exists an optimal constant $C_\theta^*$ such that 
	\begin{equation*}
	\|\nabla u\|_D^2 \geq C_{\theta}^*\|u\|_D^2 - \theta\|u\|_{\partial D}^2,
	\end{equation*}
	and hence, we obtain a lower bound for left hand side of \eqref{aim}
	\begin{equation*}
	\begin{aligned}
	&\ (2\|\nabla u\|_D^2 + 2\lambda \|u\|_{\partial D}^2 + 2\beta \|u\|_{\partial D}^2 - \alpha^2\|u\|_{\partial D}^2)\|U(s)\|_H^2 + (2-\kappa)\alpha^2\|u\|_{\partial D}^4\\
	\geq&\ (2C_{\theta}^*\|u\|_D^2 + (2\lambda + 2\beta - \alpha^2 - 2\theta)\|u\|_{\partial D}^2)(\|u\|_D^2 + \|u\|_{\partial D}^2) + (2-\kappa)\alpha^2\|u\|_{\partial D}^4\\
	=&\ 2C_{\theta}^*\|u\|_D^4 + (2C_{\theta} + 2\lambda + 2\beta - \alpha^2 - 2\theta)\|u\|_{D}^2\|u\|_{\partial D}^2 + (2\lambda + 2\beta  - 2\theta + (1 - \kappa)\alpha^2)\|u\|_{\partial D}^4.
	\end{aligned}
	\end{equation*}
	Setting $X = \|u\|_D^2$ and $Y = \|u\|_{\partial D}^2$, we see that \eqref{aim} holds 
	if we can find $\alpha$ and $\theta$ fulfilling
	\begin{equation*}
	2C_{\theta}^*X^2 + (2C_{\theta}^* + 2\lambda + 2\beta - \alpha^2 - 2\theta)XY + (2\lambda + 2\beta - 2\theta 
	+ (1 - \kappa)\alpha^2))Y^2 > 2\beta(X+Y)^2,
	\end{equation*}
	or equivalently,
	\begin{equation*}
	2(C_\theta^* - \beta)X^2 + (2C_{\theta}^* + 2\lambda - 2\beta - \alpha^2 - 2\theta)XY + (2\lambda - 2\theta + (1 - \kappa)\alpha^2)Y^2 > 0
	\end{equation*}
	for all $X, Y>0$. 
	Moreover, since $\kappa\in(0,1)$ was arbitrary, it suffices that 
	\begin{equation*}
	2(C_\theta^* - \beta)X^2 + (2C_{\theta}^* + 2\lambda - 2\beta - \alpha^2 - 2\theta)XY + (2\lambda - 2\theta + \alpha^2)Y^2 > 0
	\end{equation*}
	for all $X, Y >0$, which is equivalent to the condition
	\begin{equation}\label{Z}
		2(C_\theta^* - \beta)Z^2 + (2C_{\theta}^* + 2\lambda - 2\beta - \alpha^2 - 2\theta)Z + (2\lambda - 2\theta + \alpha^2) > 0
	\end{equation}
	for all $Z>0$. To simplify notations we introduce
	\begin{equation*}
		A = C_\theta^* - \beta \quad \text{ and } \quad B = \theta - \lambda,
	\end{equation*}
	and rewrite  \eqref{Z} as  
	\begin{equation}\label{z0}
		2A Z^2 + (2A - 2B - \alpha^2)Z + \alpha^2 - 2B > 0 \qquad \forall  Z >0.
	\end{equation}
	Since $aX^2 + bX + c > 0\  \forall X>0 $  if and only if 
	\begin{equation*}
		 a>0, \quad c\geq 0 \quad \text{ and } \quad  b \geq -2\sqrt{ac},
	\end{equation*}
	\eqref{z0} is equivalent to the following set of conditions
	\begin{align}
			2A &> 0,\label{z1} \\
			\alpha^2 - 2 B &\geq 0,\label{z2}\\
			2A - 2B - \alpha^2 &> -2\sqrt{2A (\alpha^2 - 2B)}\label{z3}.
	\end{align}

	\begin{itemize}
		\item If $A >2B$ we distinguish two different cases.
		\begin{itemize}
			\item If $\alpha^2 < 2A - 2B$ then \eqref{z3} is automatically satisfied. Hence, together with \eqref{z2} 
			we obtain the following admissible range for $\alpha$
			\begin{equation*}
				B \leq \frac{\alpha^2}{2} < A - B.
			\end{equation*}
			\item If $\alpha^2 \geq 2A - 2B$, then taking the square of both sides of \eqref{z3} implies that 
			\begin{equation}\label{z4}
				\alpha^4 - 4(3A - B)\alpha^2 + 4(A+B)^2 < 0.
			\end{equation}
			A necessary condition for the existence of a solution $\alpha$ is, that the discriminant is positive, i.e.
			\begin{equation*}
				(-2(3A - B))^2 - 4(A +B)^2 = 32A (A - B) > 0,
			\end{equation*}
			which is true due to $A > 2B$ and \eqref{z1}. Hence, we solve \eqref{z4} and get
			\begin{equation*}
				Z_1 < \frac{\alpha^2}{2} < Z_2,
			\end{equation*}
			where
			\begin{equation}\label{Z12}
				Z_{1,2} := 3A - B \mp 2\sqrt{2A(A - B)}.
			\end{equation}
		\end{itemize}
		It is easy to check that $B < Z_1 < A - B$ and $Z_2 > A - B$. Therefore, in the case $A > 2B$ 
		the admissible range of noise intensities determined by \eqref{z1}--\eqref{z3} is
		\begin{equation*}
			B \leq \frac{\alpha^2}{2} < Z_2.
		\end{equation*}
		
		\item If $A \leq 2B$, then  \eqref{z2} implies that $\alpha^2 \geq 2B \geq A + (A-2B) = 2(A - B)$. 
		Moreover,  taking the square of both sides of \eqref{z3} we again obtain inequality \eqref{z4}. 
		Necessary for the existence of a solution $\alpha$ is that $A > B$, and hence, it follows that 
		\begin{equation*}
			Z_1 < \frac{\alpha^2}{2} < Z_2,
		\end{equation*}
		with $Z_{1,2}$ defined in \eqref{Z12}.
	\end{itemize}
	In conclusion, from  \eqref{aim} we derive the following admissible ranges for the noise intensity $|\alpha|$:
		\begin{itemize}
				\item $\text{If } \quad  A > 2B  \quad \text{ then } \quad B \leq \dfrac{\alpha^2}{2} < Z_2.$
				\item $\text{If } \quad 2B \geq A > B  \quad\text{ then } \quad Z_1 < \dfrac{\alpha^2}{2} < Z_2$.
		\end{itemize}
	This concludes the proof of Theorem \ref{pro}.
\end{proof}

Theorem \ref{pro} provides sufficient conditions for the stabilisation by noise on the boundary. 
Whether for a given domain $D$ and parameters $\beta$ and $\lambda$ a suitable $\theta$ 
exists crucially depends  on the optimal constant $C_\theta^*$. However, since $C_\theta^*$ 
is not explicitly known, these conditions are hard to verify. We 
remark that Theorem 3.2 remains valid if $C_\theta^*$ is replaced by another constant $C_\theta$ 
satisfying \eqref{func_ineq}. Hence, in the following we use the expression for
the constant $C_\theta$ in Lemma \ref{TrIneq} to derive explicit conditions for the
exponential stability of the zero solution that are determined by $\alpha, \beta, d$ and 
the diameter of the domain $D$. We recall that in Lemma \ref{TrIneq} for every $\theta\geq 0$ the constant $C_\theta$ is given by
\begin{equation*}
	C_\theta = \begin{cases}
				\theta\left(d/R - \theta\right) &\text{ if } \theta \in \left[0,\frac{d}{2R}\right),\\
				\frac{d^2}{4R^2} &\text{ if } \theta \in \left[\frac{d}{2R},\infty\right),
				\end{cases}
\end{equation*}
where $R>0$ is such that  $2R = \mathrm{diam}(D)$. 
Note that $C_\theta \leq d^2/4R^2$ for all $\theta\geq 0$. 

\begin{theorem}\label{main}\hfil
\begin{itemize}
	\item[(a)] Persistence of stability: \ 
	If $\beta < C_{\lambda}$,  
	the zero steady state is exponentially stable for all noise
	intensities $|\alpha|$ such that
	$$
	\frac{\alpha^2}{2} \in \left[0, (3+2\sqrt{2})(C_\lambda-\beta)\right).
	$$
	\item[(b)] Stabilisation by noise: In case $\beta \geq C_\lambda$, if
	\begin{equation}\label{cond_bl}
		0 < \lambda < \frac 12\left(d/R - 1\right) \quad \text{ and } \quad \beta < \frac 14\left(d/R -1\right)^2 + \lambda,
	\end{equation}
	then there exists $\theta$ such that
	\begin{equation}\label{cond}
		C_\theta - \beta > \theta - \lambda >0,
	\end{equation}
	and consequently,  the zero solution of 
		\eqref{eq} is exponentially stable for all noise intensities $|\alpha|$ such that 
	\begin{equation*}
		\frac{\alpha^2}{2}\in
		\begin{cases}
			[\theta-\lambda, Z_2) &\text{ if }\  C_\theta - \beta > 2(\theta - \lambda),\\
			(Z_1, Z_2) &\text{ if }\  2(\theta - \lambda) \geq C_\theta - \beta,
		\end{cases}
	\end{equation*}
	where $Z_{1,2}$ are defined in \eqref{Z12} with $C_\theta$ in place of $C_\theta^*$.
\end{itemize}

\begin{remark}
	For the condition \eqref{cond_bl} to hold it is necessary that $d > R$, which means that we can 
	only show stabilisation if the diameter of the domain is not too large in comparison to the dimension.
\end{remark}
\begin{proof}
	For (a) we can choose $\theta = \lambda$ in Theorem \ref{pro}, and replace $C_\theta^*$ by $C_\theta$ to obtain the desired range of $\alpha$. 
	
	From Theorem \ref{pro}, again with $C_\theta$ in place of $C_\theta^*$, we know that if \eqref{cond} is satisfied 
	then the equation can be stabilised. To ensure that the set of parameters $\lambda$ and $\beta$ satisfying \eqref{cond_bl} 
	is not empty, we first observe that \eqref{cond_bl} implies that $\beta<\frac{d^2}{4R^2}$.
	Moreover, we need that
	$$
	\frac{1}{4}(d/R - 1)^2 + \lambda > C_\lambda = \lambda(d/R - \lambda),
	$$ 
	since $\lambda < \frac{1}{2}(d/R - 1) < d/2R$. 
	But this condition is equivalent to $(\lambda -\frac{1}{2}(d/R - 1))^2 > 0,$ which is obviously true due to \eqref{cond_bl}.

	Our goal now is to show that the hypotheses \eqref{cond_bl} imply the existence of some $\theta \leq d/2R$ satisfying \eqref{cond}. Indeed, since $\theta \leq d/2R$, we use that $C_\theta = \theta(d/R - \theta)$ by Lemma \ref{TrIneq} and rewrite the condition \eqref{cond} as
	\begin{equation*}
		\begin{cases}
			d/2R \geq \theta > \lambda,\\
			\theta^2 + (1 - d/R)\theta + \beta -\lambda <0.
		\end{cases}
	\end{equation*}
	This system is equivalent to
	\begin{equation*}
		d/2R \geq \theta > \lambda \quad \text{ and } \quad \frac{d/R - 1 - \sqrt\Psi}{2} < \theta < \frac{d/R - 1+\sqrt{\Psi}}{2},
	\end{equation*}
	where
	\begin{equation*}
			\Psi:= (1-d/R)^2 - 4(\beta - \lambda)>0
	\end{equation*}
	due to the condition on $\beta$ in \eqref{cond_bl}. It is obvious that
	\[
		\frac{d}{2R} > \frac{d/R - 1 - \sqrt{\Psi}}{2}.
	\]
	Moreover, by \eqref{cond_bl} it also follows that
	\begin{equation*}
		\lambda < \frac{d/R - 1 + \sqrt{\Psi}}{2}.
	\end{equation*}
	Therefore, \eqref{cond} is satisfied for all $\theta$ within the interval
	\begin{equation*}
		\max\left\{\lambda;\frac{d/R-1+\sqrt{\Psi}}{2}\right\} < \theta < \min\left\{\frac{d}{2R}; \frac{d/R-1+\sqrt{\Psi}}{2}\right\},
	\end{equation*}
	which proves part (b).
\end{proof}
\end{theorem}
In Section \ref{sec:1D} we consider a concrete one-dimensional  example for the stabilisation by noise in (b), where the trivial solution of the deterministic problem is unstable, but adding noise on the boundary with appropriate intensity leads to stabilisation.

\medskip
\noindent{\bf Discussion.} 
\begin{itemize}
\item
The results of Theorem \ref{main} can be interpreted as follows:
\begin{itemize}
	\item 
	If $\beta < C_{\lambda}\leq C_\lambda^*$, the zero steady state of the deterministic equation is exponentially stable due to Theorem \ref{determin}. 
	Theorem \ref{main} shows that the stability  is preserved for the stochastic problem, if the intensity of the noise is not too large; more precisely, for $\frac{\alpha^2}{2} \in [0, (3+2\sqrt{2})(C_\lambda-\beta))$.
	
	\item 
	If $\beta \geq C_{\lambda}$, the zero solution of the unperturbed deterministic problem might be unstable, 
	while for $\beta > C_\lambda^*$ the zero solution is unstable by Theorem \ref{determin}.
	In this case, Theorem \ref{main} shows that for sufficiently small domains and the  parameter ranges \eqref{cond_bl}, adding 
	noise with suitable intensity implies stability. 
	In fact, by choosing $\theta>0$ such that \eqref{cond} holds,
	the zero solution of the stochastic problem is exponentially stable if the intensity of the noise is within the stated range. 
	In this case, the lower bound for the noise intensity is strictly positive.
\end{itemize}	
	\item 
	These results are essentially different from related results in the literature on the stabilisation of parabolic PDEs by noise. 
	In case of  homogeneous Dirichlet boundary conditions and 
	a multiplicative It\^o noise $\alpha u\, dW_t$ acting 
	inside the domain, the noise typically helps to stabilise the system. More precisely, the zero solution is exponentially stable 
	for all sufficiently large intensities $|\alpha|$ (see e.g., \cite{CCLR07,CKS06,Kwi99,CLM01}).
	
	With the current technique, we are able to prove the stabilisation by noise on the boundary only for a finite range of $\alpha$. 
	Whether it is always possible to stabilise equation \eqref{eq}, or whether a noise with higher intensity
	destroys the stability of the zero steady state remain interesting open questions. 
			
	\item 
	The conditions in Theorem \ref{main} essentially depend on the constant $C_\theta$ defined in Lemma \ref{TrIneq}, and thus, implicitly depend on the geometry of the domain $D$. These conditions are obviously not optimal since $C_\theta \leq C_\theta^*$. Nevertheless,  Lemma \ref{TrIneq} yields an explicit 
	expression for the constant $C_\theta$ which allows to specify parameter ranges for $\lambda$ and $\beta$ 
	for which stabilisation by noise can be shown. 
	These ranges increase with decreasing diameter of the domain, or in other words, it is easier to stabilise by boundary noise as the domain gets smaller.
	
	\item If $\beta \geq C_{\theta,\max}^*$ (see Remark \ref{remark1} (iii)), the hypotheses of Theorem \ref{pro} are never satisfied, and our method of proof does not apply. However, due to the existence of an upper bound for $C_\theta^*$, 
	we conjecture that there is indeed a threshold value  $\beta_{\mathrm{crit}}$ so that \eqref{eq} 
	cannot be stabilised by noise on the boundary if $\beta\geq\beta _{\mathrm{crit}}$.
	
	\begin{conjecture*}
		There exists $\beta_{\mathrm{crit}} > 0$ such that if $\beta \geq \beta_{\mathrm{crit}}$, 
		the zero solution of \eqref{e1} is unstable for any $\lambda>0$ and $\alpha \in \mathbb R$.
	\end{conjecture*}	
	
\end{itemize}

\section{Stabilisation by noise inside the domain}\label{sec:Stochastic2}

We now analyse whether the results on stabilisation change if the noise acts inside the domain for noise free dynamical boundary conditions. Using the same notations as in the previous section
we consider the following stochastic It\^ o problem
\begin{equation}\label{eq_s2}
		\begin{cases}
			d u + (- \Delta u + u^3 - \beta u)dt  = \alpha u\, dW_t &\text{in}\ Q_T,\\
			d u + (\partial_{\nu}u + \lambda u)dt = 0 &\text{on}\ S_T,\\
			u(x,0) = u_0(x), & x\in D,\\
			u(x,0) = \phi(x), & x\in \partial D.
		\end{cases}
\end{equation}
Similarly to \eqref{EqAb}, it can be rewritten in the abstract form
\begin{equation}\label{EqAb_s2}
	\begin{cases}
		dU + (AU + B(U))dt = \widetilde C(U)dW_t,\\
		U(0) = U_0 = (u_0, \phi)\in H, 
	\end{cases}
\end{equation}
where $A$ and $B$ were defined in the beginning of Section \ref{sec:Stochastic}.
The stochastic perturbation $\widetilde C:H\to H$ is given by $\widetilde C(U) = (\alpha u, 0)$, i.e., 
it acts on the first component of the solution, and different from \eqref{oper-estimate} we now have
\begin{equation*}
	\langle \widetilde C(U), U \rangle  = \alpha \|u\|_D^2,\qquad \forall U\in H.
\end{equation*}

The existence of strong solutions follows again from \cite[Theorem 4.1]{Par79}. 
\begin{theorem}
		Let $T>0$ and assume that the initial data $U_0 = (u_0, \phi)$ is an $H$-valued 
		$\mathcal F_0$-measurable random variable satisfying $\mathbb{E}\|U_0\|_{H}^2 <\infty$. 
		Then, there exists a unique strong solution of \eqref{EqAb_s2} in $C([0,T]; H)$ such that
		\begin{equation*}
			\mathbb{E}\Big(\sup_{t\in [0,T]}\|U(t)\|_{H}^2\Big) < \infty \quad \text{ and } \quad 
			\mathbb{E}\Big(\int_0^T\|U(t)\|_{V}^2dt\Big) < \infty.
		\end{equation*}
\end{theorem}

\begin{theorem}\label{pro1}
Let $C_\theta^*$ denote the optimal constant in Lemma \ref{TrIneq} and the initial data $U_0\in L^ 2(\Omega,\mathcal{F}_0,P;H)$ 
be such that $\|U_0\|_H\neq 0$ $P$-a.s..

	Assume that there exists $\theta> 0$ such that  
	\begin{equation}\label{c1}
		\lambda - \theta > \beta - C_\theta^*> 0,
	\end{equation}
	or
	\begin{equation*}
		\lambda>\theta \quad \text{and} \quad C_\theta^*\geq \beta.
	\end{equation*}
	Then, the solution $U(t)$ of \eqref{EqAb_s2} satisfies
	\begin{equation*}
		\limsup_{t\to\infty}\frac{1}{t}\log\|U(t)\|_H^2 < 0\qquad P\text{-a.s.}
	\end{equation*}
	for all noise  intensities $|\alpha|$ such that 
	\begin{equation*}
		\frac{\alpha^2}{2} \in
		\begin{cases}
			[\max\{0, \beta - C_\theta^*\}, T_2) &\text{ if } \quad \lambda - \theta > 2(\beta - C_\theta^*),\\
			(T_1, T_2) &\text{ if } \quad 2(\beta - C_\theta^*) \geq \lambda - \theta > \beta - C_\theta^*,
		\end{cases}
	\end{equation*}
	where 
	\begin{equation*}
		T_{1,2} := C_\theta^* - \beta + 3(\lambda - \theta) \mp 2\sqrt{2(\lambda-\theta)(\lambda - \theta + C_\theta^* - \beta)}.
	\end{equation*}
\end{theorem}

\begin{proof} 

Following the same arguments as in the proof of Theorem  \ref{pro}, by applying the 
It\^o formula to $\log \|U(t)\|_H^2$ and using the exponential martingale inequality we obtain
\begin{equation*}
\begin{aligned}
	&\quad \log \|U(t)\|_H^2\\
	&\leq \log \|U(0)\|_H^2 + 4\frac{\log k}{\kappa} -\int_0^t\frac{2\langle AU+B(U), U\rangle - \| \widetilde C(U)\|_H^2}{\|U\|_H^2}ds- (2-\kappa)\int_0^t\frac{\langle  \widetilde C(U), U\rangle^2}{\|U\|_H^4}ds\\
	&\leq \log \|U(0)\|_H^2 + 4\frac{\log k}{\kappa}\\
	&\quad +\int_0^t\frac{2\|U\|_H^2(-\|\nabla u\|_D^2 - \lambda\|u\|_{\partial D}^2+\beta\|u\|_D^2) + \alpha^2\|u\|_D^2\|U\|_H^2 - (2-\kappa)\alpha^2\|u\|_D^4}{\|U\|_H^4}ds.
\end{aligned}
\end{equation*}
Now, we need that 
\begin{equation*}
	2\|U\|_H^2(-\|\nabla u\|_D^2 - \lambda\|u\|_{\partial D}^2+\beta\|u\|_D^2) + \alpha^2\|u\|_D^2\|U\|_H^2 - (2-\kappa)\alpha^2\|u\|_D^4 < 0,
\end{equation*}
where $\kappa \in (0,1)$. Since $\kappa\in(0,1)$ is arbitrary, and  
using the functional inequality $\|\nabla u\|_D^2 + \theta \|u\|_{\partial D}^2 \geq C_\theta^* \|u\|_D^2$ in Lemma \ref{TrIneq}, 
it is sufficient that 
\begin{equation}\label{y0}
	[2(C_\theta^* - \beta) + \alpha^2]\|u\|_D^4 + [2(C_\theta^* - \beta + \lambda -\theta) - \alpha^2]\|u\|_D^2\|u\|_{\partial D}^2 + 2(\lambda - \theta)\|u\|_{\partial D}^4 > 0.
\end{equation}
To shorten notations we set $A = C_\theta^* - \beta$ and $B = \lambda - \theta$.  
As in the proof of Theorem \ref{pro} we conclude that \eqref{y0} is equivalent to the set of 
conditions
\begin{align*}
	B &> 0,\\
	\alpha^2 + 2A &\geq 0,\\
	2(A+B)-\alpha^2 &> -2\sqrt{2B(\alpha^2 + 2A)}.
\end{align*}
We solve this system of inequalities analog (with $(-B,A)$ instead of $(A,B)$) as in the proof of Proposition \ref{pro}, 
in fact, by replacing the constants we obtain the following admissible ranges of noise intensities:
		\begin{itemize}
				\item $\text{If } \quad  B > -2A  \quad \text{ then } \quad \max\{0, -A\} \leq \dfrac{\alpha^2}{2} < T_2.$
				\item $\text{If } \quad -2A \geq B > -A  \quad\text{ then } \quad T_1 \leq \dfrac{\alpha^2}{2} < T_2$.
		\end{itemize}
		Here,
	\begin{equation*}
		T_{1,2} := A + 3B \mp 2\sqrt{2B(A +B)}.
	\end{equation*}
This completes the proof of Theorem \ref{pro1}.
\end{proof}

As in the previous section we now deduce from Theorem \ref{pro1} 
some explicit sufficient conditions for the persistence of stability and stabilization by noise for problem \eqref{eq_s2}, 
using the explicit expression for the constant $C_\theta$ in Lemma \ref{TrIneq}. We remark again that Theorem \ref{pro1} is 
valid for any constant fulfilling inequality \eqref{func_ineq} in place of $C_\theta^*$.

\begin{theorem}\label{thm_2}
Let $C_\theta$ denote the explicit constant in Lemma \ref{TrIneq}, i.e. $C_\theta = \theta(d/R -\theta)$ for $\theta \leq d/2R$ and $C_\theta = d^2/(4R^2)$ for $\theta \geq d/2R$.
\begin{itemize}
	\item[(a)] Persistence of stability: \ If $\beta < C_\lambda$, then the 
	 trivial solution of 
	\eqref{eq_s2} is exponentially stable for all noise intensities $|\alpha|$ such that 
	\begin{equation*}
		\frac{\alpha^2}{2} \in [0, 8(\lambda - \wh\theta\,)),
	\end{equation*}
	where $\wh \theta < \lambda$ is the unique constant
		such that $\wh \theta+C_{\wh\theta}=\beta+\lambda$.

	\item[(b)] Stabilisation by noise I: \  If $\beta<\lambda$, then  problem \eqref{eq_s2} can be stabilised by noise. 
	More precisely, a range of noise intensities that stabilise the equation is given by 
	\begin{equation*}
		\frac{\alpha^2}{2} \in \begin{cases}
			[\beta ,T_2) &\text{ if } \  2\beta<\lambda,\\
			(T_1, T_2) &\text{ if } \  \beta<\lambda\leq 2\beta,
		\end{cases}
	\end{equation*}
	where
	\begin{equation*} 
		T_{1,2} = 3\lambda-\beta\mp2\sqrt{2\lambda(\lambda-\beta)}.
	\end{equation*}

	\item[(c)] Stabilisation by noise II: \   In case $\beta>\max\{\lambda, C_\lambda\}$, if
	\begin{equation}\label{up_beta}
		\beta \leq \frac{1}{4}(d/R-1)^2 + \lambda,
	\end{equation}
	and $d>R,$ then there exists $\theta >0$ such that
	\begin{equation}\label{cond_bl1}
		\lambda - \theta > \beta - C_\theta \geq 0,
	\end{equation}
	and consequently, the zero solution of 
	\eqref{eq_s2} is exponentially stable for all noise intensities $|\alpha|$ such that 
		\begin{equation*}
		\frac{\alpha^2}{2} \in \begin{cases}
			[\beta - C_\theta, T_2) &\text{ if } \  \lambda - \theta > 2(\beta - C_\theta),\\
			(T_1, T_2) &\text{ if } \  2(\beta - C_\theta) \geq \lambda - \theta > \beta - C_\theta,
		\end{cases}
	\end{equation*}
	where
	\begin{equation*} 
		T_{1,2} := C_\theta - \beta + 3(\lambda - \theta) \mp 2\sqrt{2(\lambda - \theta)(\lambda - \theta + C_\theta-\beta)}.
	\end{equation*}
\end{itemize}
\end{theorem}
\begin{proof}
	To prove part (a) we observe that $T_1=0$ in Theorem \ref{pro1} if and only if $C_\theta+\theta=\lambda+\beta$. 
	Moreover, if $\beta<C_\lambda$, then $\lambda+\beta=\lambda+C_\lambda-\varepsilon$ for some $\varepsilon>0$, and 
	since $\theta+C_\theta$ is strictly increasing, there exists a unique $\wh\theta<\lambda$ such that 
	$\wh\theta+C_{\wh\theta}=\lambda+C_\lambda-\varepsilon=\beta+\lambda$. In this case, we obtain $T_2= 8(\lambda - \wh\theta)$
	in Theorem \ref{pro1},
	which implies the stated range of noise intensities.
	
	\medskip
	Part (b): If $\beta<\lambda$  we can choose $\theta$ in Theorem \ref{pro1} arbitrarily small such that \eqref{c1} holds, since $C_0^* = 0$ and $C_\theta^*$ depends continuously on $\theta$. 
	Moreover, we observe that 
	$$
	T_{1,2}|_{\theta=0}=3\lambda-\beta\mp\sqrt{2\lambda(\lambda-\beta)},
	$$
	and the statement is a direct consequence of Theorem \ref{pro1}.
	
	\medskip
	For Part (c), we first observe by direct computation that 
	$$\max\{\lambda, C_\lambda\} < \frac 14(d/R-1)^2 + \lambda,$$
	thus, the set of parameters $\beta$ satisfying $\max\{\lambda, C_\lambda \}< \beta < (1/4)(d/R - 1)^2 + \lambda$ 
	is not empty. We now prove that  \eqref{up_beta} implies that  there exists $\theta \leq d/2R$ 
	such that \eqref{cond_bl1} holds. 
	The claim then follows from Theorem \ref{pro1}.
	 If $\theta \leq d/2R$, then  $C_\theta = \theta(d/R - \theta)$, and condition \eqref{cond_bl1} is equivalent to
	\begin{equation}\label{in_sys}
		\begin{cases}
			\theta < \min\{d/2R,\lambda\},\\
			\theta^2 - (d/R - 1)\theta + \beta - \lambda \leq 0.
		\end{cases}
	\end{equation}
	Moreover, the second inequality in \eqref{in_sys} is equivalently to
	\begin{equation*}
		\frac{d/R-1-\sqrt{\Phi}}{2} \leq \theta \leq \frac{d/R-1+\sqrt{\Phi}}{2},
	\end{equation*}
	where	
	\begin{equation*}
		\Phi:= (d/R - 1)^2 - 4(\beta - \lambda) \geq 0
	\end{equation*}
	thanks to \eqref{up_beta}. Note that we also need $R>d$ for the two bounds for $\theta$ to be positive. It now suffices to check that
	\begin{equation*}
		\frac{d/R - 1 - \sqrt{\Phi}}{2} < \min\{d/2R, \lambda\}.
	\end{equation*}
	The inequality $\frac{d/R-1-\sqrt{\Phi}}{2} < d/2R$ is obviously satisfied. On the other hand, we observe that 
	\begin{equation*}
		\frac{d/R - 1 - \sqrt{\Phi}}{2} < \lambda\quad \Longleftrightarrow\quad d/R - 1 - 2\lambda < \sqrt{\Phi},
	\end{equation*}
	which certainly holds if $\lambda > \frac 12(d/R -1)$. 
	If $\lambda \leq \frac 12(d/R - 1)$ taking the square of both sides of the inequality leads to 
	$\beta > \lambda(d/R - \lambda)$, which is satisfied since we have $\beta > \max\{\lambda; C_\lambda\}$.
	
	In conclusion, under the condition $\max\{\lambda, C_\lambda \} < \beta <\frac{1}{4}(d/R - 1)^2 + \lambda$, there exists $0 < \theta \leq d/2R$ such that \eqref{cond_bl1} holds. Therefore, applying Theorem \ref{pro1} we obtain Part (c).
\end{proof}

\medskip
\noindent{\bf Discussion.} 
\begin{itemize}
	\item
The results in Theorem \ref{thm_2} can be interpreted as follows:
\begin{itemize}
	\item If $\beta < C_\lambda \leq C_\lambda^*$  the trivial steady state of the 
	deterministic equation is exponentially stable by Theorem \ref{determin}. Theorem \ref{thm_2} proves that this stability is preserved 
	for the perturbed problem \eqref{eq_s2} if the  
	noise intensity $|\alpha|$ is such that  $\alpha^2/2 \in [0,8( \lambda-\wh\theta\,))$.
	\item If $\beta \geq C_\lambda$, the zero solution of the deterministic equation might be unstable. In particular, if $\beta > C_\lambda^*$, then 
	the zero steady state is unstable by Theorem \ref{determin}.
	Theorem \ref{thm_2} shows that it is possible to stabilise the equation by noise inside the domain
	 if either (by part (b)) $\beta < \lambda$ or (by part (c))
	 $$\max\{\lambda; C_\lambda\} < \beta \leq \frac 14(d/R - 1)^2 + \lambda.$$
\end{itemize}
	\item 
Remarkably different from the case of noise acting on the boundary \eqref{eq}, where we conjecture that there exists a critical 
value $\beta_{\mathrm{crit}}$ such that the equation cannot be stabilised 
if $\beta \geq \beta_{\mathrm{crit}}$, Theorem \ref{thm_2} implies that
if $\lambda > \beta$ then, no matter how large $\beta$ is, there always exists a range of noise intensities 
that stabilise the equation \eqref{eq_s2}.

\item In Theorem \ref{thm_2}, as well as in Theorem \ref{main},
we obtain a finite range of noise intensities that stabilise the equation. 
This indicates that the dynamical boundary conditions cause difficulties when trying to stabilise the deterministic 
problem by noise (acting either inside the domain or on the boundary) with too high intensities. 
\end{itemize}

\section{The case of one dimension}\label{sec:1D}

As shown in Section \ref{sec:Deterministic}, the zero solution of the deterministic equation 
\eqref{DeEq} is exponentially stable if $C_\lambda^* > \beta$, while instability occurs when $C_{\lambda}^* < \beta$. 
However, since $C_\lambda^*$ is not explicit, the latter case is not easy to verify in practice.
In this section, we analyse system \eqref{DeEq} in one dimension. 
Firstly, we consider the linearised equation around zero, which can be solved explicitly by
using separation of variables. For a particular choice of parameters, it follows that the trivial 
solution of the linearised problem is unstable. Since all eigenvalues of the stationary problem have non-zero real part, we can apply an infinite dimensional version of the Hartman-Grobman theorem, see e.g. \cite{Lu91} or \cite[Corollary 5.1.6]{Hen81}, to conclude the instability of the nonlinear problem \eqref{DeEq}. 
Finally, we apply Theorem \ref{main}
to prove that the equation can be stabilised by noise on the boundary and 
determine a concrete range of noise intensities that imply the exponential stability of the zero steady state.

In the one-dimensional domain $D=(0,1)$ we consider the following linearisation of \eqref{DeEq} 
\begin{equation}\label{1d-equa}
	\begin{cases}
		u_t(x,t) - u_{xx}(x,t) - \beta u(x,t) = 0, &(x,t)\in (0,1)\times (0,\infty),\\
		u_t(0,t) - u_x(0,t) + \lambda u(0,t) = 0, &t\in (0,\infty),\\
		u_t(1,t) + u_x(1,t) + \lambda u(1,t) = 0, &t\in (0,\infty),\\
		u(x,0) = u_0(x), &x\in (0,1),\\
		u(0,0)=\phi_1,\quad u(1,0)=\phi_2,
	\end{cases}
\end{equation}
where $u_0 \in L^2(0,1)$ and $\phi_1,\phi_2\in\R$. 
Defining $\wh{u}(x,t)=e^{-\beta t}u(x,t)$ we observe that $\wh{u}$ satisfies the system 
\begin{equation}\label{1d-eq-new}
\begin{cases}
	 \wh{u}_t - \wh{u}_{xx} = 0,&(x,t)\in (0,1)\times (0,\infty),\\
	 \wh{u}_t(0,t) - \wh{u}_{x}(0,t) + (\beta + \lambda)\wh{u}(0,t) = 0, &t\in (0,\infty),\\
	\wh{u}_t(1,t) + \wh{u}_{x}(1,t) + (\beta + \lambda)\wh{u}(1,t) = 0, &t\in (0,\infty),\\
	\wh u{(x,0)}  = u_0(x),&x\in (0,1),\\
	\wh{u}(0,0) = \phi_1, \quad \wh{u}(1,0) = \phi_2.
\end{cases}
\end{equation}

\medskip
We will solve \eqref{1d-eq-new} explicitly using separation of variables.
We are looking for solutions of the form
\begin{equation*}
	\wh u(x,t) = X(x)T(t),
\end{equation*}
and inserting this ansatz into  \eqref{1d-eq-new} leads to
\begin{equation*}
	\frac{T'}{T}= \frac{X''}{X} = k
\end{equation*}
for some constant $k\in \mathbb R$. This immediately implies that
\begin{equation*}
	T(t) = c_0e^{kt},
\end{equation*}
for some constant $c_0\neq 0.$
To determine $X$, we distinguish three cases.

\medskip
\noindent{\bf Case 1:} $k>0$. From $X'' = kX$ we obtain
	\begin{equation*}
		X(x) = c_1e^{\sqrt k x} + c_2e^{- \sqrt k x},
	\end{equation*}
	for some constants $c_1, c_2 \in \mathbb R$. Consequently,
	\begin{equation*}
		\wh u(x,t) = c_3e^{kt}\left(e^{\sqrt k x} + c_4e^{-\sqrt k x}\right),
	\end{equation*}
	for some constants $c_3, c_4 \in \mathbb R$, where $c_3 \not=0$. The dynamical boundary conditions in \eqref{1d-equa} now lead to
	\begin{align*}
			(\beta + k + \lambda + \sqrt k )c_2 &= \sqrt k - \beta - k - \lambda,\\
			(\beta + k + \sqrt{k} + \lambda)e^{\sqrt k } &= c_2(\sqrt k - \beta - k - \lambda)e^{-\sqrt{k}},
	\end{align*}
	and it follows that 
	\begin{equation*}
		(\beta + k + \lambda + \sqrt{k})^2e^{\sqrt k } = (\beta + k + \lambda - \sqrt{k})^2e^{-\sqrt k}.
	\end{equation*}
	This is impossible since $\beta, \lambda$ and $k$ are positive.
	
\medskip
\noindent{\bf Case 2:} $k=0$. In this case we obtain  
	\begin{equation*}
		T(t) = c_0 \quad \text{ and } \quad X(x) = c_1x + c_2,
	\end{equation*}
	for some constants $c_0,c_1,c_2\in\R.$
	Using again the boundary conditions it follows that 
	\begin{align*}
			(\beta + \lambda) c_2&= c_1,\\
			\beta(c_1+ c_2) +c_1 + \lambda(c_1+ c_2) &= 0,
	\end{align*}
	which implies that $c_1 =  c_2 = 0$, and consequently, $u \equiv 0$.
	
\medskip
\noindent{\bf Case 3:} $k < 0$. We set $k = -\mu^2$, where $\mu \in (0,\infty)$. From $X'' = -\mu^2 X$ we obtain 
	\begin{equation*}
		X(x) = c_1\cos(\mu x) + c_2\sin(\mu x),
	\end{equation*}
	for some constants $c_1,c_2\in\R,$ and thus, 
	$$
		\wh u(x,t) = c_0e^{- \mu^2t}(c_1\cos(\mu x) + c_2\sin(\mu x)).
	$$
	Using the dynamical boundary conditions it follows that 
	\begin{align*}
			(\beta + \lambda - \mu^2)c_1 - \mu c_2 &= 0,\\
			[(\beta + \lambda - \mu^2)\cos \mu - \mu \sin \mu]c_1 + [\mu \cos \mu + (\beta + \lambda - \mu^2)\sin \mu]c_2&= 0.
	\end{align*}
	For the existence of a nontrivial solution $(c_1, c_2)$ it is necessary that 
	\begin{equation*}
		\det\begin{bmatrix}
			(\beta + \lambda - \mu^2) & -\mu\\
			(\beta + \lambda - \mu^2)\cos \mu - \mu \sin \mu & \mu \cos \mu + (\beta + \lambda - \mu^2)\sin \mu
		\end{bmatrix} = 0.
	\end{equation*}
	Hence, if $\cos \mu =0$, i.e. $\mu=\frac{\pi}{2} + k\pi$ for $k\in\N$, then $\mu^2 = (\beta + \lambda - \mu^2)^2$, which yields two 
	positive roots $\mu = \pm\frac12 + \sqrt{\frac14+ \beta+\lambda}>0$, which require 
	$\beta+\lambda = -\frac14 + (\pm \frac12 + \frac{\pi}{2} + k\pi)^2$ for some $k\in\N$. Otherwise, if  
	$\cos \mu \neq 0$ then
	\begin{equation}\label{mu-eq}
		\tan \mu = \frac{2\mu^3 - 2(\beta + \lambda)\mu}{\mu^4 - [2(\beta + \lambda)+1]\mu^2 + (\beta + \lambda)^2}.
	\end{equation}
	It is easy to see that this equation possesses countably infinitely many positive solutions 
	$$0< \mu_1 \leq \mu_2 \leq \ldots \leq \mu_n\leq \dots,\qquad \lim_{n\to\infty} \mu_n= \infty.$$

\medskip
In conclusion, the solution of \eqref{1d-eq-new} can be written as
	\begin{equation*}
		\begin{aligned}
		\wh u(x,t) &= \sum_{n=1}^{\infty}c_{0,n}e^{- \mu_n^2t}\left(c_{1,n}\cos(\mu_nx) + c_{2,n}\sin(\mu_nx)\right)\\
		&= \sum_{n=1}^{\infty}c_{n}e^{ - \mu_n^2t}\left(\cos(\mu_nx) + \frac{\beta + \lambda - \mu_n^2}{\mu_n}\sin(\mu_nx)\right),
		\end{aligned}
	\end{equation*}
	where the coefficients $c_n$ are determined by the initial data. 
	Finally, changing back to $u(x,t)$ we obtain the explicit solution to \eqref{1d-eq-new}
	\begin{equation*}
		u(x,t) = \sum_{n=1}^{\infty}\frac{\langle U_0, \varphi_n\rangle}{\|\varphi_n\|^2_H} e^{(\beta - \mu_n^2)t}\varphi_n(x),
	\end{equation*}
	where $U_0=(u_0,(\phi_1,\phi_2))$ and
	$$
	\varphi_n(x)=\cos(\mu_nx) + \frac{\beta + \lambda - \mu_n^2}{\mu_n}\sin(\mu_nx).
	$$
	From this, we observe that the zero solution of \eqref{1d-equa} is unstable if the parameters $\beta$ and $\lambda$ are such that
	\begin{equation}\label{instab}
		\beta > \mu_1^2.
	\end{equation}

	\medskip
	If $\beta + \lambda \ll 1$, we can use an asymptotic analysis to determine an explicit condition for instability.
	\begin{proposition}[Instability for small $\beta + \lambda$]\label{unstable} There exists $\varepsilon>0$ small enough such that, if $\beta + \lambda \leq \varepsilon$ and $\beta > 2\lambda$, then the zero steady state of \eqref{1d-equa} is unstable.
	\end{proposition}
	\begin{proof}
		For simplicity we set $b = \beta + \lambda$. Note that when $b \to 0$ the smallest positive solution of \eqref{mu-eq} also converges to $0$. Therefore, for small $b$, $\mu_1$ is expected to be close to $0$. In a general asymptotic expansion $\mu^2 = \gamma b^\sigma$ of \eqref{mu-eq}, direct computations verify that only the exponent $\sigma =1$ leads to a significant degeneration. Thus, by using the ansatz $\mu^2 = \gamma b$ and a Taylor expansion for the function $\tan\mu = \mu + \mu^3/3 + O(\mu^5)$ we obtain from \eqref{mu-eq}
		\begin{equation}\label{b-eq}
			1 + \frac{\gamma b}{3} + O(b^2) = \frac{2(\gamma - 1)}{b(\gamma-1)^2 - \gamma}.
		\end{equation} 
		By expanding $\gamma = \gamma_0 + \gamma_1b + O(b^2)$, the left hand side of \eqref{b-eq} is $1 + \gamma_0b/3 + O(b^2)$. Let $f(b)$ denote the right hand side of \eqref{b-eq}, then a Taylor expansion yields
		\begin{equation*}
			f(b) = f(0) + f'(0)b + O(b^2) = \frac{2(1-\gamma_0)}{\gamma_0} - \frac{2\gamma_1+2(\gamma_0-1)^3}{\gamma_0}b + O(b^2).
		\end{equation*}
		By identifying the zero and the first order terms of $b$ in \eqref{b-eq}, it follows that
		\begin{equation*}
			1 = \frac{2(1-\gamma_0)}{\gamma_0} \quad \text{ and } \quad \frac{\gamma_0}{3} = -\frac{2\gamma_1+2(\gamma_0-1)^3}{\gamma_0}
		\end{equation*}
		which implies that $\gamma_0 = 2/3$ and $\gamma_1 = -1/27$. Therefore, asymptotically we have
		\begin{equation*}
			\mu_1^2 = \frac{2}{3}b - \frac{1}{27}b^2 + O(b^3) = \frac{2}{3}(\beta + \lambda) - \frac{1}{27}(\beta+\lambda)^2 + O((\beta+\lambda)^3).
		\end{equation*}
		It follows that, when $\beta + \lambda$ is small enough, we have $\mu_1^2 < \frac{2}{3}(\beta + \lambda) < \beta$ since $\beta > 2\lambda$. This confirms the instability of the zero solution of \eqref{1d-equa} due to \eqref{instab}.
	\end{proof}
	\medskip
	In the next proposition we formulate a concrete example where the deterministic equation is unstable and by applying Theorem \ref{main} we derive an explicit range of noise intensities on the boundary which stabilise the equation.
	\begin{proposition}
		We consider the equation \eqref{1d-equa} with $\beta = 0.02$, $\lambda = 0.001$. 
		Then, the zero solution is unstable, but can be stabilised by a multiplicative It\^ o noise on the boundary. 
		
		In particular, the zero steady state of the stochastic problem 
		\begin{equation}\label{1d-equa-noise}
			\begin{cases}
			u_t(x,t) - u_{xx}(x,t) + u^3(x,t) - \beta u(x,t) = 0, &(x,t)\in (0,1)\times (0,\infty),\\
			u_t(0,t) - u_x(0,t) + \lambda u(0,t) =\alpha u(0,t) dW_t, &t\in (0,\infty),\\
			u_t(1,t) + u_x(1,t) + \lambda u(1,t) =\alpha u(1,t)dW_t, &t\in (0,\infty),\\
			u(x,0) = u_0(x), &x\in (0,1),\\
		u(0,0)=\phi_1,\quad u(1,0)=\phi_2
			\end{cases}
		\end{equation}
		is exponentially stable for all $\alpha$ such that 
		\begin{equation}\label{alpha_range}
			0.0556 < \alpha^2 <6.274.
		\end{equation}
	\end{proposition}
	\begin{proof}
		Thanks to Proposition \ref{unstable}, the zero solution of 
		the linearised equation for \eqref{1d-equa-noise} is unstable. Since all the eigenvalues 
		of the operator $A$ are positive, we can apply an infinite dimensional version of the 
		Hartman-Grobman theorem, see e.g. \cite{Lu91} or \cite[Corollary 5.1.6]{Hen81}, to 
		conclude that the zero solution of the nonlinear problem \eqref{1d-equa-noise} is also unstable.
		
		To prove the second statement of the proposition, we apply Theorem \ref{main}. 
		We need to find $\theta>0$ such that
		\begin{align*}
				C_\theta - \beta&> \theta - \lambda > 0.
		\end{align*}
		Since $D = (0,1)$ we have $R=1/2$ and $d/2R = 1$, and consequently, by Lemma \ref{TrIneq},  
		$$
		C_\theta = \begin{cases}
		\theta(2-\theta) & \theta < 1\\ 1 &\theta \geq 1.
		\end{cases}
		$$ 
		This leads to the conditions
		\begin{align*}
		\theta&>\lambda = 0.001,\\
				0&>\theta^2 - \theta + 0.019,
		\end{align*}
		and solving the inequality we obtain
		\begin{equation*}
			0.0194 < \theta < 0.9806.
		\end{equation*}
		By Theorem \ref{main} for any $\theta$ within this range, 
		the zero solution of equation \eqref{1d-equa-noise} is exponentially stable 
		if the noise intensity is such that  $Z_1 < \alpha^2/2 < Z_2$, where
		\begin{equation*}
			Z_{1,2}(\theta) = -3\theta^2 + 5\theta - 0.059 \mp 2\sqrt{2(\theta^2-2\theta+0.02)(\theta^2-\theta+0.019)}.
		\end{equation*}
		We observe that 
		\begin{equation*}
			\min_{\theta \in (0.0194, 0.9806)}\{Z_1(\theta)\} = 0.0278 \qquad \text{ and }
			\qquad  \max_{\theta \in (0.0194, 0.9806)}\{Z_2(\theta)\} =3.137,
		\end{equation*}
		and thus, obtain the desired range of intensities \eqref{alpha_range}.
	\end{proof}

\section{Conclusion}\label{sec:concl}
This paper studies the stabilisation effect of noise on the boundary for a Chafee-Infante equation with dynamical boundary conditions. Under certain conditions of the domain and the parameters, we show that with suitable boundary noise one can stabilise the trivial stationary state, which is unstable in the deterministic case, i.e. without noise. The main tools are the refinements of the ideas in \cite{CLM01} and the functional inequality: for each $\theta>0$ there exists an optimal constant $C_\theta^*>0$ such that
\begin{equation*}
	\int_D|\nabla u(x)|^2dx + \theta\int_D |u(x)|^2dS(x) \geq C_\theta^*\int_D|u(x)|^2dx \quad \text{ for all } u\in H^1(D).
\end{equation*}
A main difference to related results in the literature on stabilisation by noise (within a domain) is the fact that 
we prove a \emph{finite amplitude range of stabilising noise}.
It is however an open question whether such results are due to technical limitation or the nature of stabilisation by boundary noise. However, Theorem~\ref{thm_2} seems to indicate that the finite range of stabilisation might be due to the nature of dynamical boundary conditions. 

Up to the best of our knowledge, this work proves the first results concerning the question of stabilisation for partial differential equations using boundary noise. The Chafee-Infante equation and the dynamical boundary condition are chosen as a specific example to present the main ideas. As future research, we expect that the results of this paper generalise e.g. to Dirichlet or Neumann boundary conditions.

\medskip
\noindent{\bf Acknowledgements.} 
 We would like to thank Tom\'as Caraballo for valuable comments and remarks, and
 Max Winkler and Quoc-Hung Nguyen for useful discussions 
concerning the Poincar\'e-Trace Inequality \eqref{crucial_ineq}.

This work was initiated during the visit of the forth author to the University of Graz, 
the University's hospitality is greatly acknowledged.

Our research was supported by the ASEAN-European Academic University Network (ASEA-UNINET), and partially supported by NAWI Graz, IGDK 1754, and NAFOSTED project 101.01-2017.302.


\begin{thebibliography}{00}

	\bibitem{Adm} R. A. Adams, Sobolev Spaces, \textit{Academic Press} (2003).

	\bibitem{AB02} E. Al\`os, S. Bonaccorsi. Stochastic partial differential equations with Dirichlet white-noise boundary conditions. \textit{Annales de l'Institut Henri Poincare (B) Probability and Statistics} 38.2 (2002) 125--154.

	\bibitem{AB02a} E. Al\`os and S. Bonaccorsi., Stability for stochastic partial differential equations with Dirichlet white-noise boundary conditions.
	\textit{Infin. Dimens. Anal. Quantum Probab. Relat. Top.} 5 (2002) 465--481
	
	\bibitem{Bar11} V. Barbu. Stabilization of Navier--Stokes Flows. \textit{Springer, London}, 87--175.
	
	\bibitem{Car06} T. Caraballo, 
	Recent results on stabilization of PDEs with noise, 
	\textit{Bol. Soc. Esp. Mat. Apl.} 37 (2006), 47--70.
	
	\bibitem{CCLR07} T. Caraballo, H. Crauel, J. Langa, J. Robinson, 
	The effect of noise on the Chafee-Infante equation: a nonlinear case study, 
	\textit{Proceedings of the American Mathematical Society} 135 (2007), 373--382.
	
	\bibitem{CJ86} H.S. Carslaw, J.C. Jaeger,
	Conduction of heat in solids, Second Edition, 
	\textit{Oxford Science Publications} (1986).
	
	\bibitem{CK09} T. Caraballo, P.E. Kloeden,
	Stabilization of evolution equations by noise, 
	Interdiscip. Math. Sci., 8, World Sci. Publ., Hackensack, NJ (2010), 43--66.
	
	\bibitem{CKS06} T. Caraballo, P.E. Kloeden, B. Schmalfu\ss, 
	Stabilization of stationary solutions of evolution equations by noise, 
	\textit{Discrete Conts. Dyn. Systems, Series B.} 6 (2006), 1199--1212.
	
	\bibitem{CLM01} T. Caraballo, K. Liu, X.R. Mao, 
	On stabilization of partial differential equations by noise, 
	\textit{Nagoya Math. J.} 161(2) (2001), 155--170.
	
	\bibitem{CMR17} R. Czaja, P. Mar\' in-Rubio, 
	Pullback exponential attractors for parabolic equations with dynamical boundary conditions, 
	\textit{Taiwanese J. Math.}  21  (2017), 819--839.
	
	\bibitem{DFT07} A. Debussche, M. Fuhrman, G. Tessitore. Optimal control of a stochastic heat equation with boundary-noise and boundary-control. \textit{ESAIM: Control, Optimisation and Calculus of Variations}. 13(1) (2007) 178--205.
	
	\bibitem{Esc93} J. Escher, 
	Quasilinear parabolic systems with dynamical boundary conditions, 
	\textit{Comm. Partial Differential Equations} 18  (1993), 1309--1364.
	
	\bibitem{Fil14} A. Filinovskiy, 
	On the eigenvalues of a Robin problem with a large parameter, 
	\textit{Mathematica Bohemica}, 139 (2014), 341--352.
	
	\bibitem{Fil15} A. Filinovskiy, 
	On the Asymptotic Behavior of the First Eigenvalue of Robin Problem With Large Parameter, 
	\textit{J. Elliptic Para. Equations}, 1 (2015), 123--135.
	
	\bibitem{GM10} L. Gawarecki, V. Mandrekar, 
	Stochastic differential equations in infinite dimensions with applications to stochastic differential equations, 
	\textit{Springer-Verlag Berlin Heidelberg} 2011.
	
	\bibitem{Hen81} D. Henry, 
	Geometric Theory of Semilinear Parabolic Equations, 
	\textit{Springer-Verlag} Lecture Notes in Mathematics (1981).
	
	\bibitem{Kov14} H. Kova\v{r}\'ik, 
	On the Lowest Eigenvalue of Laplace Operators with Mixed Boundary Conditions, 
	\textit{J. Geom. Anal.} 24 (2014), 1509--1525.
	
	\bibitem{God06} G.R. Goldstein, 
	Derivation and physical interpretation of general boundary conditions, 
	\textit{Adv. Differential Equations} 11 (2006), 457--480.
	
	\bibitem{Kwi99} A.A. Kwiecinska, 
	Stabilization of partial differential equations by noise, 
	\textit{Stoch. Proc.\& Appl.} 79 (1999), 179--184.
	
	\bibitem{Li69} J.-L. Lions, 
	Quelques méthodes de résolution des problèmes aux limites non linéaires. (French) 
	Dunod; Gauthier-Villars, Paris 1969.
	
	\bibitem{Li97} K. Liu, 
	On stability for a class of semilinear stochastic evolution equations, 
	\textit{Stochastic Process. Appl.} 70 (1997), 219--241.
	
	\bibitem{Lu91} K. Lu, 
	A Hartman-Grobman theorem for scalar reaction-diffusion equations, 
	\textit{J. Diff. Eqs.} 93 (1991), 364--394.
	
	\bibitem{Mao94} X.R. Mao, 
	Stochastic stabilization and destabilisation, 
	\textit{Systems \& Control Letters} 23 (1994), 279--290.
	
	\bibitem{Mun17} I. Munteanu. Stabilization of stochastic parabolic equations with boundary-noise and boundary-control. \textit{Journal of Mathematical Analysis and Applications}. 449(1) (2017), 829--842.
	
	\bibitem{Par79} E. Pardoux, 
	Stochastic partial differential equations and filtering of diffusion processes, 
	\textit{Stochastic Process. Appl.} 3 (1979), 127--167.
	
	\bibitem{SM09} M. Sofonea, A. Matei, 
	Variational inequalities with applications: a study of antiplane frictional contact problems, 
	Vol. 18, \textit{Springer Science \& Business Media}, 2009.
	
	\bibitem{Sow94} R.B. Sowers. Multidimensional reaction-diffusion equations with white noise boundary perturbations. \textit{The Annals of Probability} (1994) 2071--2121.
\end{thebibliography}
\end{document}